\documentclass[11pt]{amsart}
\usepackage[margin=1.35in]{geometry}
\usepackage{amscd,amsmath,amsxtra,amsthm,amssymb,stmaryrd,xr,mathrsfs,mathtools,enumerate,commath, comment}
\usepackage{stmaryrd}
\usepackage{xcolor}
\usepackage{commath}
\usepackage{comment}
\usepackage{tikz-cd}
\usepackage{longtable} 
\usepackage{pdflscape} 
\usepackage{booktabs}
\usepackage{hyperref}
\definecolor{vegasgold}{rgb}{0.77, 0.7, 0.35}
\definecolor{darkgoldenrod}{rgb}{0.72, 0.53, 0.04}
\definecolor{gold(metallic)}{rgb}{0.83, 0.69, 0.22}
\hypersetup{
 colorlinks=true,
 linkcolor=darkgoldenrod,
 filecolor=brown,      
 urlcolor=gold(metallic),
 citecolor=darkgoldenrod,
 pdftitle={Diophantine equations of the form $Y^n=f(X)$ over function fields},
 }

\usepackage[all,cmtip]{xy}

\DeclareFontFamily{U}{wncy}{}
\DeclareFontShape{U}{wncy}{m}{n}{<->wncyr10}{}
\DeclareSymbolFont{mcy}{U}{wncy}{m}{n}
\DeclareMathSymbol{\Sh}{\mathord}{mcy}{"58}
\usepackage[T2A,T1]{fontenc}
\usepackage[OT2,T1]{fontenc}

\newtheorem{theorem}{Theorem}[section]

\newtheorem{lemma}[theorem]{Lemma}

\newtheorem*{theorem*}{Theorem}
\newtheorem*{corollary*}{Corollary}

\newtheorem{definition}[theorem]{Definition}

\numberwithin{equation}{section}

\theoremstyle{remark}
\newtheorem{remark}[theorem]{Remark}

\newcommand{\Z}{\mathbb{Z}}

\newcommand{\F}{\mathbb{F}}

\newcommand{\cO}{\mathcal{O}}

\newcommand{\op}[1]{\operatorname{#1}}

\begin{document}
\title[Diophantine equations of the form $Y^n=f(X)$ over function fields]{Diophantine equations of the form $Y^n=f(X)$ over function fields}

\author[A.~Ray]{Anwesh Ray}
\address[A.~Ray]{Department of Mathematics\\
University of British Columbia\\
Vancouver BC, Canada V6T 1Z2}
\email{anweshray@math.ubc.ca}

\begin{abstract}
Let $\ell$ and $p$ be (not necessarily distinct) prime numbers and $F$ be a global function field of characteristic $\ell$ with field of constants $\kappa$. Assume that there exists a prime $P_\infty$ of $F$ which has degree $1$, and let $\cO_F$ be the subring of $F$ consisting of functions with no poles away from $P_\infty$. Let $f(X)$ be a polynomial in $X$ with coefficients in $\kappa$. We study solutions to diophantine equations of the form $Y^{n}=f(X)$ which lie in $\cO_F$, and in particular, show that if $m$ and $f(X)$ satisfy additional conditions, then there are no non-constant solutions. The results obtained apply to the study of solutions to $Y^{n}=f(X)$ in certain rings of integers in $\mathbb{Z}_{p}$-extensions of $F$ known as \emph{constant $\Z_p$-extensions}. We prove similar results for solutions in the polynomial ring $K[T_1, \dots, T_r]$, where $K$ is any field of characteristic $\ell$, showing that the only solutions must lie in $K$. We apply our methods to study solutions of diophantine equations of the form $Y^n=\sum_{i=1}^d (X+ir)^m$, where $m,n, d\geq 2$ are integers. 
\end{abstract}

\subjclass[2020]{}

\maketitle
\section{Introduction}

\par Let $n\geq 2$ be an integer and $\kappa$ be a finite field of characteristic $\ell>0$. Let $F$ be a global function field with field of constants $\kappa$ and assume that there exists a prime $P_\infty$ of $F$ of degree $1$. In other words, we assume that there is a prime $P_\infty$, which is totally inert in the composite $\bar{\kappa}\cdot F$. The ring of integers $\cO_F$ shall consist of all functions $f\in F$ with no poles away from $P_\infty$. Given a polynomial $f(X)$ with coefficients in $\kappa$, we study solutions to the diophantine equation $Y^n=f(X)$ for which both $X$ and $Y$ lie in $\cO_F$. More precisely, we prove that if certain additional conditions are met, then there are no non-constant solutions, i.e., $X$ and $Y$ must both belong to the field of constants $\kappa$. Equations of the form $Y^n=f(X)$ are of significant interest. We shall apply our analysis to study a class of diophantine equations which involve perfect powers in arithmetic progressions. Let $m, n, d\geq 2$ be integers and let $r\geq 1$, then, there has been significant interest in the classification of integral solutions to the diophantine equation 
\[Y^n=(X+r)^m+(X+2r)^m+\dots+(X+dr)^m,\] cf. for instance, \cite{bennett2004diophantine,bai2013diophantine, hajdu2015conjecture, bennett2017perfect,berczes2018diophantine, koutsianas2018perfect, patel2018perfect, argaez2020perfect}.

\par We also explore themes motivated by the Iwasawa theory of function fields. Mazur initiated the Iwasawa theory of elliptic curves over number fields (cf. \cite{mazur1972rational}), which had applications to the study of growth of Mordell--Weil ranks of elliptic curves in certain infinite towers of number fields. One hopes to extend such lines of investigation to curves of higher genus (cf. \cite{ray2022rational}) and more generally, study the stability and growth of solutions to any diophantine equation in an infinite tower of global fields. In this paper, we shall study certain function field analogues of such questions, however, instead of elliptic curves, we consider the class of diophantine equations of the form $Y^n=f(X)$, where $f(X)$ has constant coefficients. Let us explain our results in greater detail. Given any integer $n$, there is a unique extension $\kappa_n/\kappa$ such that $\op{Gal}(\kappa_n/\kappa)$ is isomorphic to $\Z/n\Z$. Given a prime $p$, set $\kappa_n^{(p)}$ to denote $\kappa_{p^n}$ and set $F_n^{(p)}$ to denote the composite $F\cdot \kappa_n^{(p)}$. This gives rise to a tower of function field extensions 
\[F=F_0^{(p)}\subset F_1^{(p)}\subset \dots \subset F_n^{(p)}\subset F_{n+1}^{(p)}\subset \dots. \] Let $\Z_p$ denote the ring of $p$-adic integers, i.e., the valuation ring of $\mathbb{Q}_p$.
The \emph{constant $\Z_p$-extension} of $F$ is the infinite union \[F_\infty^{(p)}:=\bigcup_{n} F_n^{(p)}.\]It is easy to see that the Galois group $\op{Gal}\left(F_\infty^{(p)}/F\right)$ is isomorphic to $\Z_{p}$. Let $h_F$ denote the class number of $F$ (cf. \cite[Chapter 5]{rosen2002number}). Note that since $P_\infty$ is assumed to have degree $1$, it remains inert in $F_n^{(p)}$ for all $p$. Let $\cO_{\infty}^{(p)}$ (resp. $\cO_{n}^{(p)}$) be the ring of integers of $F_\infty^{(p)}$ (resp. $F_n^{(p)}$), i.e., the functions $f\in F_\infty$ with no poles away from $P_\infty$. We now state the main result.

\begin{theorem*}[Theorem \ref{thm 2.5}]
Let $\ell$ be a prime number and $F$ be a global function field of characteristic $\ell$. Let $\kappa$ be the field of constants of $F$, and let $p$ and $q$ be prime numbers that are not necessarily distinct. Assume that $q\neq \ell$. Let $f(X)$ be a polynomial with coefficients in $\kappa$ satisfying following conditions.
\begin{enumerate}
    \item The polynomial $f(X)$ factorizes into 
    \[f(X)=a_0(X-a_1)^{n_1}\dots (X-a_t)^{n_t},\] where $a_0\in \kappa$ and $a_1,\dots, a_t$ are distinct elements in $\kappa$, $n_1,\dots, n_t$ are positive integers and $t\geq 2$.
    \item At least two of the exponents $n_i$ are not divisible by $q$.
\end{enumerate} 

Then, the following assertions hold.

\begin{enumerate}
    \item\label{p1 main} Suppose that $p$ and $q$ are distinct. Then, for all sufficiently large numbers $k>0$, the only solutions $(X,Y)$ to $Y^{q^k}=f(X)$ that are contained in $\cO_\infty^{(p)}$ are constant.
    \item\label{p2 main} Suppose that $p\nmid h_F$, then the only solutions $(X,Y)$ to $Y^p=f(X)$ that are contained in $\cO_\infty^{(p)}$ are constant.
\end{enumerate}

\end{theorem*}

As a consequence of the above result, we find that for any prime $p$, there are only finitely many numbers $n$, that are not powers of $\ell$, for which $Y^n=f(X)$ has solutions in $F_\infty^{(p)}$.
A more specific criterion applies to any function field $F$, cf. Theorem \ref{thm 2.3}. The methods used in proving the above result are applied to study another question of independent interest. Let $K$ be a field of positive characteristic $\ell$ and $A$ be the polynomial ring $K[T_1, \dots, T_n]$.

\begin{theorem*}[Theorem \ref{thm 4.1}]
With respect to above notation, let $f(X)$ be a polynomial with all of its coefficients and roots in $K$. Let $q\neq \ell$ be a prime number and assume that the following conditions are satisfied.
\begin{enumerate}
    \item $f(X)$ factorizes into 
    \[f(X)=a_0(X-a_1)^{n_1}\dots (X-a_t)^{n_t},\] where $a_0\in K$ and $a_1,\dots, a_t$ are distinct elements in $K$, $n_1,\dots, n_t$ are positive integers and $t\geq 2$.
    \item At least two of the exponents $n_i$ are not divisible by $q$.
\end{enumerate} 
Then, any solution $(X,Y)\in A^2$ to \begin{equation}\label{Y power k}
    Y^{q}=f(X)
\end{equation} is constant, i.e., $X$ and $Y$ are both in $K$.
\end{theorem*}
It follows from the above result that if $n>1$ is not a power of $\ell$, then $X^n=f(X)$ does not have non-constant solutions in $A$.
\par \emph{Organization:} Including the introduction, the manuscript consists of $5$ sections. In section \ref{s 2}, we prove criteria for the constancy of solutions to $Y^n=f(X)$ in global function fields $F$. The main result proven in section \ref{s 2} is Theorem \ref{thm 2.3}. In section \ref{s 3}, we extend the results in section \ref{s 2} to prove the constancy of solutions to the above equation in $\Z_p$-extensions of $F$. It is in this section that we prove the main result of the paper, i.e., Theorem \ref{thm 2.5}. In section \ref{s 4}, we prove similar results for the polynomial rings over a field. Finally, in section \ref{s 5}, we study the specific case when $f(X)=\sum_{i=1}^k (X+ir)^m$. 

\section{Constancy of solutions to $Y^n=f(X)$ in a global function field}\label{s 2}
\par In this section, we introduce basic notions and prove results about the solutions to certain diophantine equations over global function fields. Throughout this section, $\ell$ be a prime number and $A$ be an integral domain of characteristic $\ell$ with field of constants $\kappa$. We introduce the notion of a \emph{discrete valuation} on $A$.
\begin{definition}\label{d def}A function $d:A\rightarrow \Z$ is said to be a discrete valuation if the following conditions are satisfied.
\begin{enumerate}
    \item \label{d def 1} The values taken by $d$ are non-negative.
    \item Let $\mathbf{1}$ be the identity element of $A$, we have that $d(\mathbf{1})=0$.
    \item Given non-zero elements $f,g\in A$, we have that $d(fg)=d(f)+d(g)$, 
    \item $d(f+g)\leq \op{max}\{d(f),d(g)\}$,
    \item\label{d def 5} suppose that $d(f)<d(g)$, then $d(f+g)=d(f)$. 
\end{enumerate}
\end{definition}
Let $A_0$ be the subring of $A$ consisting of all elements $a\in A$ for which $d(a)\leq 0$. Given $f,g\in A$, we say that $f$ divides $g$ if $fh=g$ for some $h\in A$. It is clear that if $f$ divides $g$ then $d(f)\leq d(g)$. 
\begin{lemma}\label{lemma 2.1}
Let $q$ be a prime number such that $q\neq \ell$, and $A$ be an integral domain of characteristic $\ell$ equipped with a function $d$ satisfying the conditions \eqref{d def 1} to \eqref{d def 5} of Definition \ref{d def}. Let $f,g,c\in A$ satisfy the equation \begin{equation}\label{f g equation}f^q-g^q=c.\end{equation} Then, we have that $d(f), d(g)\leq d(c)$. In particular, $f$ and $g$ are contained in $A_0$ if $c$ is contained in $A_0$.
\end{lemma}
\begin{proof}
Suppose by way of contradiction that $d(f)>d(c)$, or $d(g)>d(c)$. Assume first that $d(f)>d(c)$. Set $e:=(g-f)$, from \eqref{f g equation}, we find that $e$ divides $c$. As a result, $d(e)\leq d(c)< d(f)$, hence by the property \eqref{d def 5}, we find that \[d(g)=d\left(f+e\right)=d(f).\] Therefore, we have deduced that $d(g)>d(c)$. Rewrite \eqref{f g equation} as \[(g+e)^q-g^q=c,\] and expand the left hand side of the above equation via the binomial expansion 
\[qeg^{q-1}+{q\choose 2} e^2 g^{q-2}+\dots +e^q=c.\] Note that since $d(e)<d(g)$, we find that for all $i$ such that $2\leq i\leq q$, \[d\left({q\choose i}e^i g^{q-i}\right)< d\left(qeg^{q-1}\right),\] and therefore, 
\[d(c)=d\left((g+e)^q-g^q\right)=d\left(qeg^{q-1}\right)=(q-1)d(g).\] This implies that $d(g)\leq d(c)$, a contradiction. On the other hand, if we assume that $d(g)>d(c)$ (instead of assuming that $d(f)>d(c)$), the same argument applies. 
\end{proof}
\par We shall illustrate the above result in various cases of interest. In this section, we study diophantine equations over global function fields $F$. Let $\ell$ be a prime number and denote by $\F_\ell$ the finite field with $\ell$ elements (i.e. $\Z/\ell\Z$). Fix an algebraic closure $\bar{F}$ of $F$. Let $\kappa$ be the algebraic closure of $\F_\ell$ in $F$, and let $\bar{\kappa}$ be the algebraic closure of $\kappa$ in $\bar{F}$. Set $F'$ to denote the composite of $F$ with $\bar{\kappa}$.
\par Following \cite[Chapter 5]{rosen2002number}, a \emph{prime} $v$ of $F$ is by definition the maximal ideal of a discrete valuation ring $\cO_v\subset F$ with fraction field equal to $F$. A divisor of $F$ is a finite linear combination $D=\sum_v n_v v$ of primes $v$. In the above sum, $n_v$ are all integers, and the set of primes $v$ for which $n_v\neq 0$ is referred to as the \emph{support} of $D$. Given a function $g\in F$, denote by $\op{div}(g)$ the associated principal divisor. Note that any principal divisor has degree $0$. Two divisors are considered equivalent if the differ by a principal divisor. The \emph{class group} of $F$ is the group of divisor classes of degree $0$, and has finite cardinality (cf. \cite[Lemma 5.6]{rosen2002number}). Denote by $h_F$ the \emph{class number}, i.e., the number of elements in the class group. Given a natural number $N$, denote by $h_F[N]$ the cardinality of the $N$-torsion in class group.

\par The field $F'$ is identified with the field of fractions of a projective algebraic curve $\mathfrak{X}$ over $\bar{\kappa}$. A point $w\in \mathfrak{X}(\bar{\kappa})$ is also referred to a prime of $F'$, since it corresponds to a valuation ring $\cO_w\subset F'$ with fraction field $F'$. Given a prime $w$ of $F'$ and a prime $v$ of $F$, we say that $w$ lies above (or divides) $v$ if the natural inclusion of fields $F\hookrightarrow F'$ induces an inclusion of valuation rings $\cO_v\hookrightarrow \cO_w$. Given a function $g\in F$ (resp. $g\in F'$), denote by $\op{ord}_v(g)$ (resp. $\op{ord}_w(g)$) the \emph{order of vanishing} of $g$ at $v$ (resp. $w$). We refer to $d_v(g):=-\op{ord}_v(g)$ (resp. $d_w(g):=-\op{ord}_w(g)$) the order of the pole of $g$ at $v$ (resp. $w$). Given a finite and nonempty set of primes $S$ of $F$, the ring of $S$-integers $\cO_S$ consists of all functions $g\in F$ such that $d_v(g)\leq 0$ for all primes $v\notin S$. Let $\bar{S}$ be the set of primes of $F'$ that lie above $S$. Let $A$ denote the composite $\cO_S\cdot \bar{\kappa}$. A function $g\in A$ satisfies the property that for all $w\notin \bar{S}$, $d_w(g)\leq 0$. According to our conventions, $\cO_F$ is the ring of $S$ integers where $S:=\{P_\infty\}$. Since $P_\infty$ is a prime of degree $1$, is totally inert in $F'$. By abuse of notation, the single prime in $\bar{S}$ is also denoted $P_\infty$. 
\par We list some basic properties of the function $d_w$ on $A$. The following result applies for any ring of $\bar{S}$-integers.

\begin{lemma}\label{lemma 2.2}
Let $f$ and $g$ be a functions in $A$ and $w$ be a point in $\mathfrak{X}(\bar{\kappa})$. Then, the following assertions hold.
\begin{enumerate}
    \item\label{p1} Suppose that $d_w(g)\leq 0$ for all $w\in \bar{S}$, then, $g$ is a constant function,
    \item\label{p2} $d_w(fg)=d_w(f)+d_w(g)$, 
    \item\label{p3} suppose that $d_w(f)>d_w(g)$, then, $d_w(f+g)=d_w(f)$.
\end{enumerate}
\end{lemma}

\begin{proof}
\par Note that since $g$ is contained in $A$, $d_w(g)\leq 0$ for all points $w\notin \bar{S}$. Therefore, the assumption that $d_w(g)\leq 0$ implies that $g$ has no poles, and thus must be a constant function. This proves part \eqref{p1}.
\par Part \eqref{p2} clearly follows from the relation $\op{ord}_w(fg)=\op{ord}_w(f)+\op{ord}_w(g)$. 
\par For part \eqref{p3}, we note that $f+g=f(1+g/f)$, and since it is assumed that $d_w(f)>d_w(g)$, it follows that $g/f$ vanishes at $w$. As a result, $d_w(1+g/f)=0$ and thus, 
\[d_w(f+g)=d_w(f)+d_w(1+f/g)=d_w(f),\] which proves the result.
\end{proof}
 Recall that $P_\infty$ is a prime of degree $1$ and $\cO_F$ is the associated ring of integers in $F$.
\begin{theorem}\label{thm 2.3}
Let $\ell$ be a prime number and $F$ be a global function field of characteristic $\ell$. Let $\cO_F$ be the ring of integers of $F$. Denote by $\kappa$ the field of constants of $F$. Let $f(X)$ be a polynomial with coefficients in $\kappa$. Let $q\neq \ell$ be a prime number and let $k>0$ be the least integer such that $h_F[q^k]=h_F[q^{k-1}]$. Assume that the following conditions are satisfied.
\begin{enumerate}
    \item $f(X)$ factorizes into 
    \[f(X)=a_0(X-a_1)^{n_1}\dots (X-a_t)^{n_t},\] where $a_0\in \kappa$ and $a_1,\dots, a_t$ are distinct elements in $\kappa$, $n_1,\dots, n_t$ are positive integers and $t\geq 2$.
    \item At least two of the exponents $n_i$ are not divisible by $q$.
\end{enumerate} 
Then, any solution $(X,Y)$ to \begin{equation}\label{Y power k}
    Y^{q^k}=f(X)
\end{equation} for which $X,Y\in \cO_F$ is constant, i.e., $X$ and $Y$ are both in $\kappa$.
\end{theorem}

\begin{proof}
Since the elements $a_1,\dots, a_t$ are distinct elements of $\kappa$, we find that for all $i, j$ such that $i\neq j$, $(X-a_i)-(X-a_j)=a_j-a_i$ is a non-zero element of $\kappa$. Therefore, for $i\neq j$, the divisors $\op{div}(X-a_i)$ and $\op{div}(X-a_j)$ have disjoint supports.

\par From the equation \eqref{Y power k}, we find that 
\begin{equation}\label{divisor sum equation}q^k \op{div}(Y)=\sum_{i=1}^t n_i \op{div}(X-a_i).\end{equation}
Since the divisors $\op{div}(X-a_i)$ in the above sum have disjoint supports, and therefore for all $i$, there are divisors $D_i'$ such that $q^k D_i'=n_i\op{div}(X-a_i)$. Recall that it is assumed that there are two distinct indices $i$ and $j$ such that $q\nmid n_i$ and $q\nmid n_j$. Without loss of generality, assume that $q\nmid n_1$ and $q\nmid n_2$. Therefore, there exist divisors $D_1$ and $D_2$ such that $q^k D_i=\op{div}(X-a_i)$ for $i=1,2$. The divisor classes $[D_1]$ and $[D_2]$ in the class group are in the $q^k$-torsion subgroup of the class group. Since $h_F[q^k]=h_F[q^{k-1}]$, we find that $q^{k-1}D_i$ is principal for $i=1,2$. Let $f$ and $g$ be functions in $F$ such that
\[\op{div}(f)=q^{k-1}D_1\text{ and } \op{div}(g)=q^{k-1} D_2.\] Thus we find that $u_1f^q=(X-a_1)$ and $u_2 g^q=(X-a_2)$, where $u_1$ and $u_2$ are contained in $\kappa$. Note that since $(X-a_i)$ has no poles away from $\{P_\infty\}$, the same is true for $f$ and $g$, hence, $f,g\in \cO_F$. Setting $A:=\cO_F\cdot \bar{\kappa}$, we may replace $f$ by $u_1^{1/q} f$ and $g$ by $u_2^{1/q} g$, and thus assume that $f^q=(X-a_1)$ and $g^q=(X-a_2)$ for some elements $f,g\in A$. We find that $f^q-g^q=a_2-a_1$ is a non-zero element of $\kappa$. The pair $(A,d_{P_\infty})$ satisfies the properties \eqref{d def 1} to \eqref{d def 5}, and therefore $d_{P_\infty}(f)=d_{P_\infty}(g)=0$ by Lemma \ref{lemma 2.1}. Therefore, by part \eqref{p1} of Lemma \ref{lemma 2.2}, we find that $f$ and $g$ are both in $\kappa$. Hence, $X=f^q+a_1$ is in $\kappa$ and thus so is $Y$.
\end{proof}

\begin{remark} We make the following observations.
\begin{itemize}
    \item Theorem \ref{thm 2.3} above implies that $k=1$ if $q\nmid h_F$.
    \item If the roots $a_i$ are not contained in $F$, we may base change $F$ by an extension $\kappa'$ of $\kappa$ which is generated by the roots $a_i$.
    \item Suppose that $f(X)$ satisfies the conditions of Theorem \ref{thm 2.3}. Then, since $q\nmid h_F$ for all but finitely many primes $q$, thus $Y^q=f(X)$ has no non-constant solutions in $F$ for all but finitely many primes $q$. In fact, it is easy to see that Theorem \ref{thm 2.3} implies that $Y^n=f(X)$ has no non-constant solutions for all but finitely many natural numbers $n$.
\end{itemize}
 
\end{remark}

\section{Constancy of solutions to $Y^n=f(X)$ in the constant $\Z_p$-extension of a function field}\label{s 3}
\par In this section, we apply Theorem \ref{thm 2.3} proven in the previous section, to study questions motivated by Iwasawa theory. Given primes $p$ and $q$ (not necessarily distinct) let $h_n(p,q)$ denote $\# \op{Cl}(F_n^{(p)})[q^\infty]$, the cardinality of the $q^\infty$-torsion in the class group of $F_n^{(p)}$. 

\begin{theorem}[Leitzel, Rosen]\label{rosen theorem}
Let $p$ and $q$ be (not necessarily distinct) prime numbers and $F$ be a function field of characteristic $\ell$. The following assertions hold
\begin{enumerate}
    \item\label{p1 rosen} Suppose that $p$ and $q$ are distinct. Then, as $n$ goes to infinity, the quantity $h_n(p,q)$ is bounded.
    \item\label{p2 rosen} Suppose that $p$ does not divide $h_F$. Then, $h_n(p,p)=1$ for all $n$.
\end{enumerate}
\end{theorem}
\begin{proof}
For part \eqref{p1 rosen}, the result follows from \cite[Theorem 11.6]{rosen2002number}. For function fields of genus $1$, the result was proven by Leitzel, cf. \cite{leitzel1970class}. For part \eqref{p2 rosen}, the reader is referred to \cite[Proposition 11.3]{rosen2002number}. 
\end{proof}

Recall notation from the introduction. The prime $P_\infty$ is totally inert in $F_\infty^{(p)}$ for any prime $p$. We set $\cO_\infty^{(p)}$ to denote the ring of integers of $F_\infty^{(p)}$, i.e., the functions with no poles outside $\{P_\infty\}$. The following is the main result of this manuscript.

\begin{theorem}\label{thm 2.5}
Let $\ell$ be a prime number and $F$ be a global function field with field of constants $\kappa$. Let $p$ and $q$ be prime numbers that are not necessarily distinct, and assume that $q\neq \ell$. Let $f(X)$ be a polynomial with coefficients in $\kappa$ satisfying the conditions of Theorem \ref{thm 2.3}. 

Then, the following assertions hold.

\begin{enumerate}
    \item\label{p1 main} Suppose that $p$ and $q$ are distinct. Then, for all sufficiently large numbers $k>0$, the only solutions $(X,Y)$ to $Y^{q^k}=f(X)$ in $\cO_\infty^{(p)}$ are constant.
    \item\label{p2 main} Suppose that $p\nmid h_F$, then the only solutions $(X,Y)$ to $Y^p=f(X)$ in $\cO_\infty^{(p)}$ are constant.
\end{enumerate}

\end{theorem}

\begin{proof}
First, we consider the case when $p=q$. It follows from part \eqref{p1 rosen} of Theorem \ref{rosen theorem} that $h_{n}(p,q)$ is bounded as $n$ goes to infinity. Let $k>0$ be such that $q^k$ be larger than $\op{max} h_n(p,q)$. It follows from Theorem \ref{thm 2.3} that $Y^{q^k}=f(X)$ has no non-constant solutions in $\cO_n^{(p)}$ for all $n$, and therefore, no non-constant solutions in $\cO_\infty^{(p)}$. Hence, there are no non-constant solutions in $F_\infty$ as well.
\par Next, we consider the case when $p=q$ and that $p\nmid h_F$. Note that if $p\nmid h_F$, then by part \eqref{p2 rosen} of Theorem \ref{rosen theorem} that $ h_n(p,p)=1$ for all $n$. It follows from Theorem \ref{thm 2.3} that $Y^p=f(X)$ has no non-constant solutions in $\cO_n^{(p)}$ for all $n$, and therefore, no non-constant solutions in $\cO_\infty^{(p)}$. 
\end{proof}
\section{Constancy of solutions to $Y^n=f(X)$ in a polynomial ring in $r$-variables}\label{s 4}

\par In this section, we study solutions to equations of the form $Y^n=f(X)$ in polynomial rings over a field. Let $K$ be \emph{any} field of characteristic $\ell>0$ and $A$ be the polynomial ring $K[T_1, \dots, T_r]$. Given a polynomial $g$, let $d_i(g)$ be the degree of $g$ viewed as a polynomial in $T_i$ over the subring $K[T_1, \dots, T_{i-1}, T_{i+1}, \dots, T_r]$. The pair $(A,d_i)$ satisfies the conditions \eqref{d def 1}--\eqref{d def 5} of Definition \ref{d def}. The class group $\op{Cl}(A)$ denotes the group of equivalence classes of Weil divisors. Since $A$ is a unique factorization domain, we have that $\op{Cl}(A)=0$.
\begin{theorem}\label{thm 4.1}
With respect to above notation, let $f(X)$ be a polynomial with all of its coefficients in $K$. Let $q\neq \ell$ be a prime number and assume that the following conditions are satisfied.
\begin{enumerate}
    \item $f(X)$ factorizes into 
    \[f(X)=a_0(X-a_1)^{n_1}\dots (X-a_t)^{n_t},\] where $a_0\in K$ and $a_1,\dots, a_t$ are distinct elements in $K$, $n_1,\dots, n_t$ are positive integers and $t\geq 2$.
    \item At least two of the exponents $n_i$ are not divisible by $q$.
\end{enumerate} 
Then, any solution $(X,Y)\in A^2$ to \begin{equation}\label{Y power k}
    Y^{q}=f(X)
\end{equation} is constant, i.e., $X$ and $Y$ are both in $K$.
\end{theorem}

\begin{proof}
Note that the algebraic closure of $K$ in $A$ is equal to $K$.
We may as well replace $K$ by its algebraic closure and assume without loss of generality the $K$ is algebraically closed and that $q\nmid n_1$ and $q\nmid n_2$. Since the class group of $A$ is trivial, the same argument as in the proof of Theorem \ref{thm 2.3} shows that $(X-a_1)=f^q$ and $(X-a_2)=g^q$ for $f,g\in A$. Therefore, we find that $f^q-g^q=a_2-a_1$, an element of $K$. Lemma \ref{lemma 2.1} then implies the $d_i(f)=d_i(g)=0$ for all $i$, hence $f,g$ are both in $K$. The result follows from this.
\end{proof}

\section{Perfect powers that are sums of powers in arithmetic progressions}\label{s 5}
In this section, we apply the results proven in previous sections to study the solutions of the diophantine equation involving perfect powers in arithmetic progressions
\begin{equation}\label{arithmetic progression}Y^n=f(X):=(X+r)^m+(X+2r)^m+\dots+ (X+dr)^m.\end{equation} Here, $m, n, r, d$ are integers such that $m, n, d\geq 2$ and $r\geq 1$. Viewing $f(X)$ as a polynomial with integral coefficients let $\Delta$ denote its discriminant.
\begin{theorem}\label{thm 5.1}
Let $F$ be a function field with characteristic $\ell\geq 5$ and field of functions $\kappa$. Let $q\neq \ell$ be a prime. Assume that the following conditions are satisfied.
\begin{enumerate}
    \item All roots of $f(X)$ are contained in $\kappa$,
    \item $\ell\nmid r$,
    \item at least one of the following conditions are satisfied:
    \begin{enumerate}
        \item $\ell\nmid \Delta$,
        \item $q>m$ and $d\not \equiv 0,\pm 1\mod{\ell}$.
    \end{enumerate}
\end{enumerate}
Let $k$ be the minimal value such that $h_F[q^k]=h_F[q^{k-1}]$. Then, there are no non-constant solutions $(X,Y)$ to 
\[Y^{q^k}=f(X)=\sum_{i=1}^d (X+ir)^m\] in $\cO_F^2$.  
\end{theorem}
\begin{proof}
Following notation from the statement of Theorem \ref{thm 2.3}, we write \[f(X)=a_0\prod_{j=1}^t (X-a_i)^{n_i},\] where $a_1,\dots, a_t$ are all distinct elements of $\kappa$. The result follows from Theorem \ref{thm 2.3} provided the following conditions are satisfied
\begin{enumerate}
    \item $t\geq 2$.
    \item At least two of the exponents $n_i$ are not divisible by $q$.
\end{enumerate} 
Note that if $\ell\nmid \Delta$, then all roots of $f(X)$ are distinct in $\kappa$, hence, $n_i=1$ for all $i$ and $t=d\geq 2$. In particular, both of the above conditions are satisfied. 
\par On the other hand, assume that $q>m$. Clearly, all values $n_i$ are less than or equal to $\op{deg} f(X)\leq m$, and since $q>m$, it follows that $q\nmid n_i$ for all $i$. It suffices to check that $t\geq 2$ if $d\not \equiv 0,\pm 1 \mod{\ell}$. Suppose not, then $f(X)$ is of the form $d(X+a)^m$, for some $a\in \kappa$.
\par Expanding $f(X)=\sum_{i=1}^d (X+ir)^m$, one obtains 
\[\begin{split}
    & \sum_{i=1}^d (X+ir)^m\\
    =& \sum_{i=1}^d \sum_{j=0}^m {m\choose j} i^{j} r^{j} X^{m-j}\\
    =& \sum_{j=0}^m {m\choose j}r^{j}\left( \sum_{i=1}^d i^{j}\right)X^{m-j}\\
    =&d X^m+mr\left(\frac{d(d+1)}{2}\right)X^{m-1}+{m\choose 2}r^2\left(\frac{d(d+1)(2d+1)}{6}\right)X^{m-2}+\dots.
\end{split}\]
Since $f(X)=d(X+a)^m$ and $\ell\nmid d$, we find that 
\[a^j=r^j \left( \frac{1}{d}\sum_{i=1}^d i^{j}\right)\] for all values of $j$. In particular, we find that 
\[a=r\left(\frac{(d+1)}{2}\right) \text{ and }a^2=r^2\left(\frac{(d+1)(2d+1)}{6}\right).\]

We thus arrive at the relation 
\begin{equation}\label{above relation}a^2=r^2\left(\frac{(d+1)}{2}\right)^2=r^2\left(\frac{(d+1)(2d+1)}{6}\right).\end{equation}
The relation holds in $\Z/\ell \Z$. Since $\ell\nmid r$ and $d\nmid -1\mod{\ell}$ by assumption, we find that the above relation \eqref{above relation} gives us 
\[\frac{(d+1)}{4}=\frac{2d+1}{6}.\] This is not possible since it is assumed that $d\not\equiv 1\mod{\ell}$. Thus, Theorem \ref{thm 2.3} applies to give the result.
\end{proof}

\begin{theorem}
Let $F$ be a function field with characteristic $\ell\geq 5$ and field of functions $\kappa$. Let $q\neq \ell$ be a prime. Assume that the following conditions are satisfied.
\begin{enumerate}
    \item All roots of $f(X)$ are contained in $\kappa$,
    \item $\ell\nmid r$,
    \item at least one of the following conditions are satisfied:
    \begin{enumerate}
        \item $\ell\nmid \Delta$,
        \item $q>m$ and $d\not \equiv 0,\pm 1\mod{\ell}$.
    \end{enumerate}
\end{enumerate}
Let $p$ be any prime number. The following assertions hold. 
\begin{enumerate}
    \item Suppose that $p\neq q$. Then for all large enough values of $k>0$, there are no non-constant solutions to $Y^{q^k}=f(X)$ in $\cO_\infty^{(p)}$.
    \item Suppose that $p=q$ and $p\nmid h_F$. Then, there are no non-constant solutions to $Y^p=f(X)$ in $\cO_\infty^{(p)}$.
\end{enumerate}
\end{theorem}
\begin{proof}
It follows from the proof of Theorem \ref{thm 5.1} that the conditions of Theorem \ref{thm 2.5} are satisfied, and thus the result follows.
\end{proof}

\begin{theorem}
Let $K$ be a field of characteristic $\ell\geq 5$ and let $A$ be the polynomial ring $K[T_1, \dots, T_r]$. Let $q\neq \ell$ be a prime. Assume that the following conditions are satisfied.
\begin{enumerate}
    \item All roots of $f(X)$ are contained in $K$,
    \item $\ell\nmid r$,
    \item at least one of the following conditions are satisfied:
    \begin{enumerate}
        \item $\ell\nmid \Delta$,
        \item $q>m$ and $d\not \equiv 0,\pm 1\mod{\ell}$.
    \end{enumerate}
\end{enumerate}
Then, any solution $(X,Y)\in A^2$ to \begin{equation}\label{Y power k}
    Y^{q}=f(X)
\end{equation} is constant, i.e., $X$ and $Y$ are both in $K$.
\end{theorem}
\begin{proof}
It follows from the proof of Theorem \ref{thm 5.1} that the conditions of Theorem \ref{thm 4.1} are satisfied. The result thus follows from Theorem \ref{thm 4.1}.
\end{proof}

\bibliographystyle{alpha}
\bibliography{references}
\end{document}